\documentclass{article}
\usepackage{float}
\usepackage{amsfonts}
\usepackage{graphicx}
\usepackage{amsmath}
\usepackage{appendix}
\usepackage{fullpage}
\usepackage{multicol}
\usepackage{hyperref}
\hypersetup{
  colorlinks   = true, 
  urlcolor     = blue, 
  linkcolor    = blue, 
  citecolor    = blue  
}

\newtheorem{theorem}{Theorem}
\newenvironment{proof}[1][Proof]{\begin{trivlist}
\item[\hskip \labelsep {\bfseries #1}]}{\end{trivlist}}

\newcommand{\bd}{\begin{displaymath}}
\newcommand{\ed}{\end{displaymath}}
\newcommand{\be}{\begin{equation}}
\newcommand{\ee}{\end{equation}}
\newcommand{\bea}{\begin{eqnarray}}
\newcommand{\eea}{\end{eqnarray}}
\newcommand{\bda}{\begin{eqnarray*}}
\newcommand{\eda}{\end{eqnarray*}}
\newcommand{\ba}{\begin{array}}
\newcommand{\ea}{\end{array}}

\newcommand{\ra}{\rightarrow}

\newcommand{\lra}{\longrightarrow}

\newcommand{\dd}{\mbox{\rm\,d}}

\begin{document}

\title{\huge{ \textbf{A theoretical connection between the noisy leaky integrate-and-fire and the escape rate models:
the non-autonomous case.
}}}

\author{Gr\'{e}gory Dumont\footnote{\'{E}cole Normale Sup\'{e}rieure,
Group for Neural Theory, Paris, France, email: gregory.dumont@ens.fr},
          Jacques Henry\footnote{INRIA team Carmen, INRIA Bordeaux Sud-Ouest, 33405 Talence cedex, France,
          email: Jacques.Henry@inria.fr}
     and
        Carmen Oana Tarniceriu\footnote{Intersdisciplinary Research Department - Field Sciences,
        Alexandru Ioan Cuza University of Ia\c{s}i, Romania
e-mail: tarniceriuoana@yahoo.co.uk}
     }


\date{}
\maketitle

\begin{abstract}
One of the most important challenges in mathematical neuroscience is to properly illustrate the stochastic nature of neurons.
Among the different approaches, the noisy leaky integrate-and-fire and the escape rate models are probably the most popular.
These two models are usually chosen to express different noise action over the neural cell. In this paper we investigate the
link between the two formalisms in the case of a neuron subject to a time dependent input. To this aim, we introduce a new
general stochastic framework. As we shall prove, our general framework entails the two already existing ones. Our result has
theoretical implications since it offers a general view upon the two stochastic processes mostly used in neuroscience, upon
the way they can be linked, and explain their observed statistical similarity.
\end{abstract}


\section{Introduction}

The noisy aspect of neuronal behavior  \cite{noiseNS} makes the use of stochastic processes unavoidable in mathematical neuroscience.
A generally accepted formalism to express the variability present in neural dynamics has not been found yet, but among the most
notable models dealing with this aspect stand out the noisy leaky integrate-and-fire (NLIF) \cite{Burkitt, Burkitt2},
and the escape rate models \cite{GvH}. Of course, the use
of each of these models comes with its own advantages.

The NLIF - a model going back to Lapique's work in 1907 \cite{B01,abbott01, Lapique} - describes neurons as simple electrical circuits
consisting in a capacitor in parallel with a resistor driven by a noisy input. Written in the language of stochastic differential
equations, it gives rise to a Langevin equation \cite{langevin} completed with a reset mechanism to mimic the onset of an
action potential.
Related to a stochastic variable satisfying the Langevin equation, stands the probability density function (pdf) that expresses
 the likelihood that the state variable takes on a specific value \cite{gardiner}. For the NLIF model,
 its associated pdf follows the well-known Fokker-Planck (FP) equation \cite{gardiner}.

The escape rate models - a formalism proposed by W. Gerstner and J.L. van Hemmen in 1992 \cite{GvH} - introduce a rather different approach, in
which the neuron's dynamics remain deterministic but the generation of an action potential occurs randomly. The firing probability
is then described by a {\em stochastic intensity of firing} or {\em hazard function} which depends on the momentary distance between
 the membrane potential and the formal firing threshold \cite{GvH, gerstner}. In this setting, the neuron can fire even without reaching the formal spiking threshold
  or remain quiescent after crossing it \cite{gerstner}. The corresponding pdf to the escape rate model satisfies a pure transport equation along
   the deterministic trajectories, and a term to express the probabilistic nature of spike initiation is added. The so-obtained pdf
   follows an age structured (AS) system, where the {\em age} variable accounts for the time elapsed since the neuron had its last
   spike \cite{GvH}.

At a first sight, the two stochastic processes reminded above seem essentially different, both in form and in modeling assumptions.
Despite their contrasting nature, it has been observed that they can behave in a surprisingly comparable fashion \cite{plesser}.
The main goal of our paper is therefore to investigate the existence of a common ground between the two formalisms.

In our previous work \cite{usdt,usit}, we made a first step in this direction. We have highlighted an unforseen relationship between
the two modeling frameworks, but, we have left several open questions, especially in the context of time dependent stimuli. To push
further our previous studies, we will now introduce a new stochastic model that describes the evolution of the membrane's potential
of a neuron and its
time passed by since the last firing event. We will show that the system satisfied by the associated pdf generates the two
previously introduced formalisms:  FP equation and AS system.

The paper is structured as follows: We shall first remind the two above mentioned models as they are usually encountered in the
literature. Next, we introduce a two dimensional stochastic description of the neural state, and, starting from there, we show that our general stochastic framework
 entails the two already existing ones. The corresponding stationary system is discussed next and a special solution of the non stationary system is proposed later on. We finish with a discussion on the biological and theoretical implications of our findings.

\section{Two standard models for neural noise}

\subsection{Fokker-Planck equation}
The first process -the NLIF- that we shall refer to
is modeled as a Langevin equation which describes the membrane potential of a neuron in the subthreshold regime.
The subthreshold dynamics are described by:
\begin{equation} \label{LIF}
\frac{\dd}{\dd t} v(t)=(\mu(t)-v) + \sigma \xi(t),
\end{equation}
where $\mu(t)$ is the mean of the external stimuli, $\sigma$ is the noise intensity of the external stimuli, and $\xi$ is a
Gaussian white noise
\begin{equation*}
\langle\xi(t)\rangle=0, \quad \langle\xi(t)\xi(t')\rangle=\delta(t-t').
\end{equation*}
The firing mechanism of the neuron is expressed by a discontinuous reset process which says that
when the membrane potential reaches the threshold value, the potential is reset to a fix given value:

\begin{equation}\label{rLIF}
\mbox{if}\quad v\geq 1 \quad\mbox{then}\quad v\mapsto v_r.
\end{equation}
By associating to the stochastic variable $v$ a pdf $p(t,v)$,
 it is found that the pdf satisfies the FP equation \cite{brunel2015fokker}:
\be\label{FP}
\frac{\partial}{\partial t} p(t,v)+\frac{\partial}{\partial v}[(\mu(t)-v)p(t,v)]
-\frac{\sigma^2}{2}\frac{\partial^2}{\partial v^2}p(t,v)=\delta(v-v_r)r(t),
\ee
where an absorbing boundary condition at the spiking threshold is imposed in order to express the immediate reset of the variable,

\begin{equation*}
p(t,1)=0,
\end{equation*}
and the flux at this boundary gives the {\em firing rate}:
\be\label{frfp}
r(t)=-\frac{\sigma^2}{2}\frac{\partial}{\partial v}p(t,1).
\ee
At the lower boundary, the no-flux condition is imposed:
\be
\lim_{v\ra -\infty} [(v-\mu(t))p(t,v)+\frac{\sigma^2}{2}\frac{\partial}{\partial v}p(t,v)]=0.
\ee
An initial repartition is assumed as known:
\be\label{FPic}
p(0,v)=p_0(v),
\ee
for which the normalization condition is imposed, due to the interpretation of $p$ as a pdf:
\be\label{CPFP}
\int_{-\infty}^1 p_0(v)\dd v=1.
\ee
It can be easily shown that, once the condition (\ref{CPFP}) is assumed, then  the conservation property
of the solution to the system (\ref{FP})-(\ref{FPic}) takes place at any time:
\be\label{CPFP1}
\int_{-\infty}^1 p(t,v)\dd v=1.
\ee

\begin{figure}[t!]
\begin{center}
    \includegraphics[width=\textwidth]{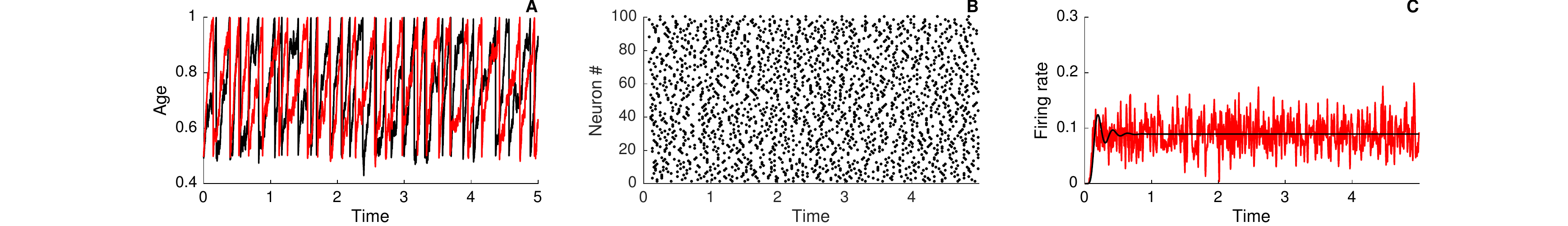}
   \caption{Simulation of the NLIF model (\ref{LIF}), (\ref{rLIF}) and of its pdf given by the FP equation (\ref{FP})-(\ref{FPic}).
    A) Illustration of the noise action onto the membrane potential via a simulation of two realizations of the stochastic
    process (\ref{LIF}), (\ref{rLIF}). B) Spiking activity generated via several realizations of the stochastic
    process (\ref{LIF}), (\ref{rLIF}).
    C) Comparison between the activity extracted from several realizations of the stochastic process (\ref{LIF}), (\ref{rLIF}) in red,
     and a simulation of its pdf given by the FP equation (\ref{FP})-(\ref{FPic}) in black. A gaussian was taken as initial condition;
      the parameters of the simulation are:  $v_r=0.5$, $\mu=3$, $\sigma =0.15$.  }\label{Fig02}
      \end{center}
\end{figure}

In Fig. \ref{Fig02}, a simulation of the NLIF model (\ref{LIF}), (\ref{rLIF}) is presented.
The first panel (Fig. \ref{Fig02}A) corresponds to the simulation of the same neuron under different noise realizations.
 As expected, the corresponding dynamics are different and the neurons
spike at different times. In the second panel (Fig. \ref{Fig02}B), the spiking activity of the same cell is extracted from a bigger
number of trials. Finally, in the last panel (Fig. \ref{Fig02}C) the corresponding spiking activity is compared with the firing rate
obtained from the FP model (\ref{FP})-(\ref{FPic}).

In Fig. \ref{Fig03}, a comparison is made under a stimulus that is time dependent.
The first panel (Fig. \ref{Fig03}A) shows the stimulus' time evolution. In the second panel (Fig. \ref{Fig03}B) the spiking
 activity of the cell is displayed, from which we can extract the firing rate. A comparison with the FP equation is
 shown in the last panel (Fig. \ref{Fig03}C).

\begin{figure}[t!]
\begin{center}
    \includegraphics[width=\textwidth]{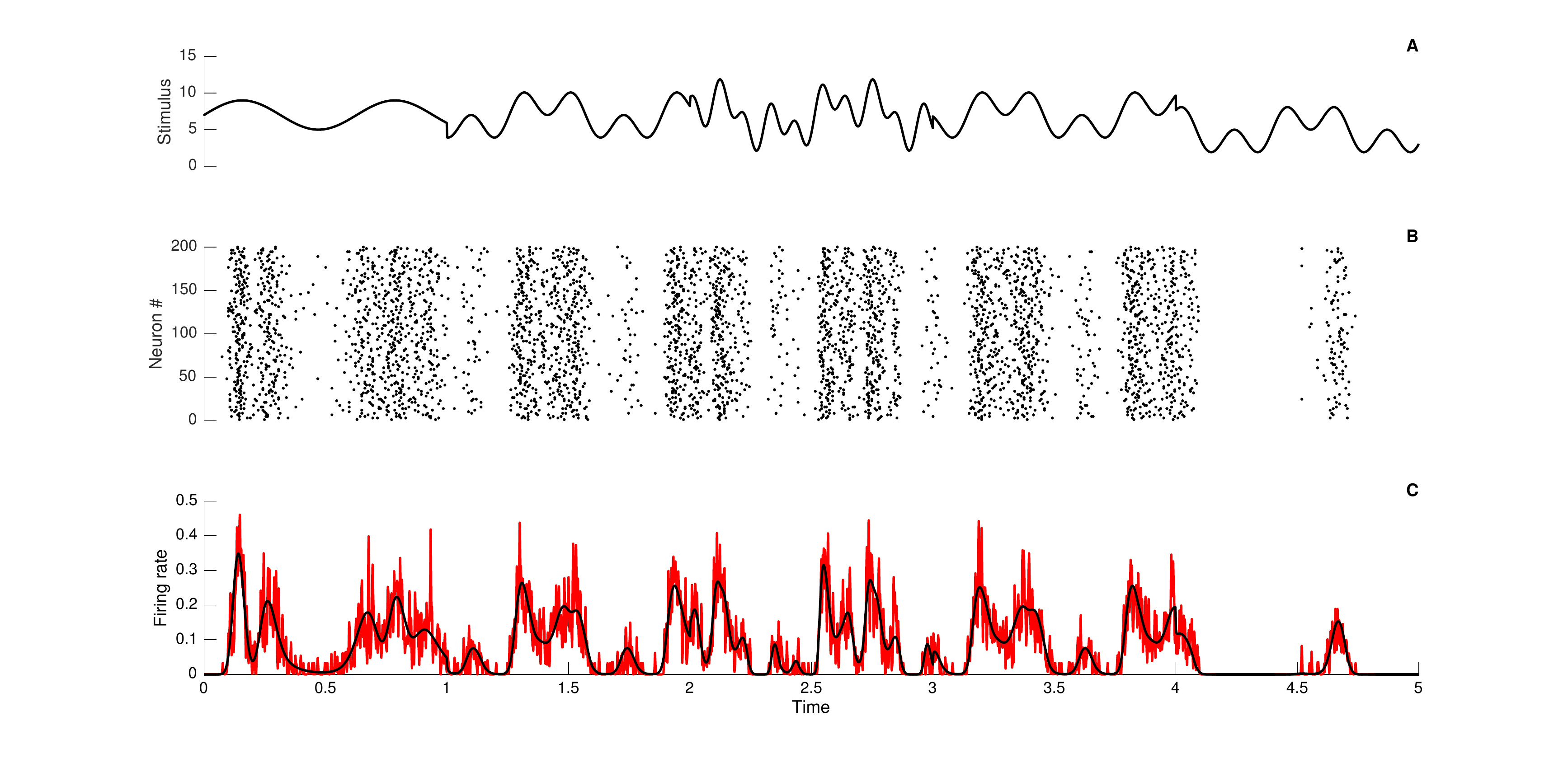}
   \caption{Comparison between the neural activity generated via realizations of the stochastic process (\ref{LIF}), (\ref{rLIF})
 and the FP model (\ref{FP})-(\ref{FPic}) for a time dependent stimulus $\mu(t)$. A) Time evolution of the stimulus.
B) Activity obtained from several realizations of the stochastic process (\ref{LIF}), (\ref{rLIF}). C) Comparison between the firing rate
obtained from the stochastic process (\ref{LIF}), (\ref{rLIF}) in red, and the one given by the  FP model (\ref{FP})-(\ref{FPic}) in black. A gaussian was taken as
initial condition; the parameters of the simulation are: $v_r=0.5$, $\sigma =0.2$.  }\label{Fig03}
      \end{center}
\end{figure}

\subsection{Age-structure formalism}
The escape rate models consider a deterministic trajectory combined with a probabilistic firing mechanism \cite{GvH}. Implicitly,
the state variable will depend only on the last firing time. Without restraining the generality, we can consider the state variable
for the deterministic trajectories as the time elapsed since the last spike, that we will generically call {\em age}:
\be\label{a}
\frac{\dd}{\dd t}a(t)=1.
\ee
The probability that an action potential occurs during a small time interval is computed via the {\em hazard function $S(t,a)$}
that gives the probability of spike of a neuron that has at time t the age $a$.
More precisely, during the time interval $(t,t+dt)$, a spike occurs with probability $S(t,a(t))dt$ and, consequently,
 the age is reset to zero immediately after:

\begin{equation}\label{r}
\mbox{ If the spike occurs at time $t$  then } a\mapsto 0.
\end{equation}
The associated probability density function, $n(t,a)$, satisfies the following PDE \cite{GvH, gerstner}:
\be\label{AS}
\frac{\partial}{\partial t}n(t,a)+\frac{\partial}{\partial a} n(t,a)+S(t,a)n(t,a)=0,
\ee
where the first two terms give the pure transport along the deterministic trajectories (\ref{a}) and the last term
accounts for the probability of spiking.

The reset mechanism is modeled as a non-local boundary condition:
\be\label{fra}
n(t,0)=r(t)=\int_0^\infty S(t,a)n(t,a)\dd a,
\ee
where the right hand side of the above equation defines the firing rate of the system.

As before, an initial repartition is
taken as known:

\be\label{ASic}
n(0,a)=n_0(a).
\ee
The conservation property of the system (\ref{AS})-(\ref{ASic}) takes place as well. Namely,
\begin{equation*}
\int_0^\infty n(t,a)\dd a=1,
\end{equation*}
for any $t$, as soon as
\begin{equation*}
\int_0^\infty n_0(a)\dd a=1.
\end{equation*}

\begin{figure}[t!]
\begin{center}
    \includegraphics[width=\textwidth]{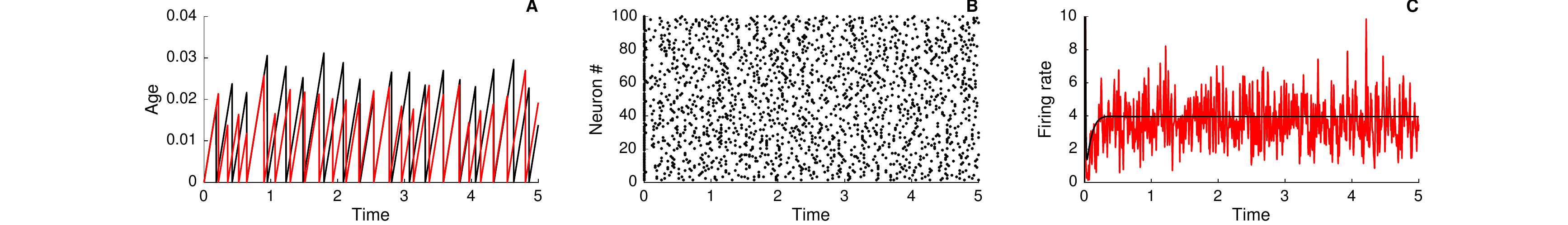}
   \caption{Simulation of the stochastic process (\ref{a}), (\ref{r}) and of its associated AS system (\ref{AS})-(\ref{ASic}).
   A) Illustration of the noise action onto the age dynamics (\ref{a}), (\ref{r}).
   B) Spiking activity generated via several realizations of
   the stochastic process (\ref{a}), (\ref{r}).
   C) Comparison between the activity extracted from several realizations of the stochastic process
   (\ref{a}), (\ref{r}) in red, and a simulation of the firing rate given by its associated AS system (\ref{AS})-(\ref{ASic}) in black.
   A gaussian was taken as initial condition; the functions used and the parameters of the simulation are:
   $S(a)=(\exp(h-V(a))$, with $V(a)=-\log(1-\exp(-a/\tau))$ and  $\tau=30$, $h=3$.  }\label{Fig05}
      \end{center}
\end{figure}

We present in  Fig. \ref{Fig05} a simulation of the model (\ref{AS})-(\ref{ASic}).
Again, we have made a comparison between the stochastic process (red curve) given by (\ref{a}), (\ref{r}) and the evolution
in time of the density function (black curve) given by the system (\ref{AS})-(\ref{ASic}).  The first panel (Fig. \ref{Fig05}A)
 corresponds to the simulation of the same neuron. As expected, due to the probabilistic firing process, the corresponding dynamics
 have different spike times.
 In the second panel (Fig. \ref{Fig05}B), the spiking activity of the same cell is extracted from a large number of trials.
 In the last panel (Fig. \ref{Fig05}C) the corresponding spiking activity is compared with the firing rate given by the AS system (\ref{fra}).

In Fig. \ref{Fig06}, a numerical simulation of the stochastic process defined
by (\ref{a}), (\ref{r}) for the same age-dependent death rate is illustrated. Note that the neuron never
fires exactly at the same age, since its probability to fire (escape) is purely
stochastic. The first panel (Fig. \ref{Fig06}A) shows the time evolution of the stimulus. In the second panel
(Fig. \ref{Fig06}B), the spiking activity of the cell is shown, from which  the firing rate is extracted.
A comparison between the firing rate directly computed and the one extracted from the AS system is presented
in the last panel (Fig. \ref{Fig06}C).

\begin{figure}[t!]
\begin{center}
    \includegraphics[width=\textwidth]{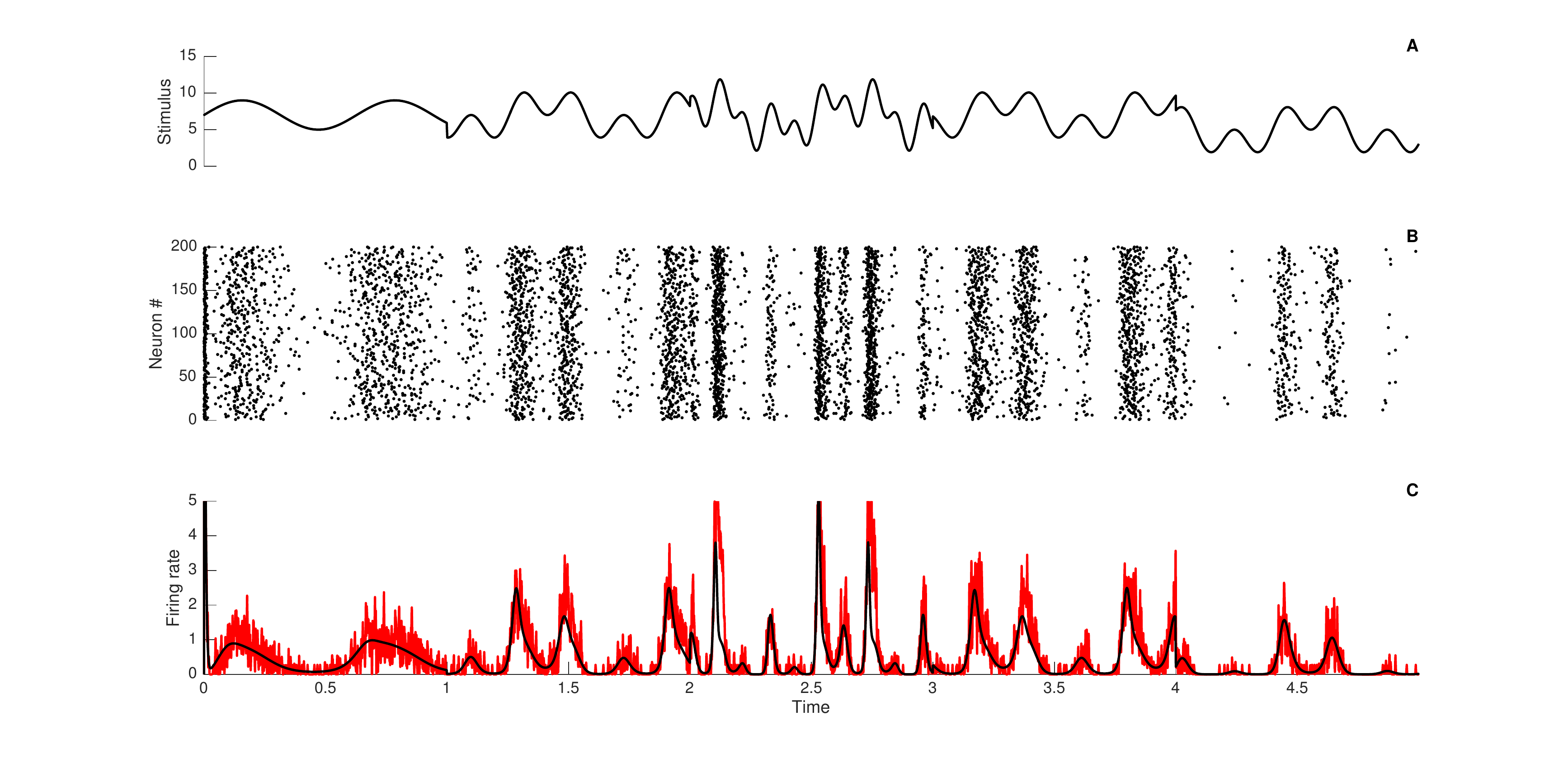}
   \caption{Comparison between the neural activity generated via realizations of the stochastic process (\ref{a}), (\ref{r})
    and the AS system (\ref{AS})-(\ref{ASic}) for a time dependent stimulus. A) Time evolution of the stimulus.
B) Activity obtained from several realizations of the stochastic process (\ref{a}).
C) Comparison between the firing rate obtained from the stochastic process (\ref{a}), (\ref{r}) in red, and the one given by the
AS system (\ref{AS})-(\ref{ASic}) in black. A gaussian was taken as initial condition; the functions and the parameters used in
 the simulation are:
  $S(t,a)=(\exp(h(t)-V(a))$, with $V(a)=-\log(1-\exp(-a/\tau))$ and  $\tau=30$. Here $h(t)$ represents the
  stimulus that is shown in the first panel. }\label{Fig06}
      \end{center}
\end{figure}

\section{An age and potential structured model. Connections with FP and AS systems}
\subsection{The age-and-potential structured model}
We shall introduce below a more general process which can be seen as a generalization of the two reminded formalisms.
Let us consider a stochastic process that models both variables, i.e. spike times and membrane potential.
More precisely, let us consider that, between two consecutive firing times, the potential variable evolves according to
the subthreshold dynamics given by the NLIF model (\ref{LIF}) and the age of the neuron grows linearly with time
according to the dynamics given by (\ref{a}). We suppose that the neuron will fire whenever a given potential threshold value
is reached. Whenever this happens, both state variables are reset: the potential is reset to $v_r$ and the age to zero.
We deal therefore with the following two-dimensional stochastic process:

\be\label{sds}
\left\{\begin{array}{ll}\frac{\dd}{\dd t} v(t)= (\mu(t)-v(t)) +\sigma \xi(t),\\
\frac{\dd}{\dd t} a(t)=1,\end{array}\right.
\ee
completed by the firing mechanism:
\be\label{firec}
\mbox{if}\quad v\geq 1 \quad\mbox{then}\quad v\mapsto v_r\quad \mbox{and}\quad a\mapsto 0.
\ee
As in the previous section, $\mu(t)$ is the mean of the external stimulus, $\sigma$ is the noise intensity
of the external stimulus, and $\xi$ is a again Gaussian white noise
\begin{equation*}
\langle\xi(t)\rangle=0, \quad
\langle\xi(t)\xi(t')\rangle=\delta(t-t').
\end{equation*}
The corresponding probability density
function satisfies a two-dimensional Fokker-Planck equation (see \cite{gardiner}):
\be\label{model2'}
\frac{\partial }{\partial t}\pi(t,a,v)+\overbrace{\frac{\partial }{\partial a}\pi(t,a,v)
+\frac{\partial }{\partial v}[(\mu(t) -v)\pi(t,a,v)]}^{\text{Drift}}
- \overbrace{\frac{ \sigma ^2}{2}     \frac{\partial^2 }{\partial v^2}\pi(t,a,v)}^{\text{Diffusion}}  = 0.
\ee
For the right choice of the boundary conditions, we turn again to the stochastic system (\ref{sds}), (\ref{firec}).
Namely, corresponding to the firing mechanism, we should impose an absorbing boundary condition at the firing threshold:
\bd
\pi(t,a,1)=0,
\ed
and a reflecting boundary at the boundary $v=-\infty$:

\bda
&&\lim\limits_{v \to -\infty} (-\mu +v)\pi(t,a,v)
+\frac{ \sigma ^2}{2}\frac{\partial }{\partial v}\pi(t,a,v) =0.\\
\eda
We will also assume that
\bd
\lim\limits_{a \to +\infty} \pi(t,a,v) =0.\\
\ed
The flux at the firing threshold is given by
\be\label{rho}
\rho(t,a)=-\frac{ \sigma ^2}{2}   \frac{\partial }{\partial v}\pi(t,a,1),
\ee
and, since once firing occurs  both membrane potential and the age of a neuron are reset in the reset value $(0,v_r)$,
 the following singular source term arises:

\be\label{fr}
\pi(t,0,v)=\delta(v-v_r)\int_0^\infty \rho(t,a)\dd a:= \delta(v-v_r) r(t).
\ee
Finally, an initial distribution is assumed as known:
\be\label{icm1}
\pi(0,a,v) = \pi_0(a,v).
\ee
The model is now complete, and one can check directly by integration over the state space that,
if the initial distribution satisfies
\begin{equation*}
\int_{-\infty}^1\int_0^\infty \pi_0(a,v)\dd a \dd v = 1,
\end{equation*}
then, for any $t>0$:
\begin{equation*}
\int_{-\infty}^1\int_0^\infty \pi(t,a,v)\dd a \dd v = 1.
\end{equation*}

\begin{figure}[]
\begin{center}
    \includegraphics[width=\textwidth]{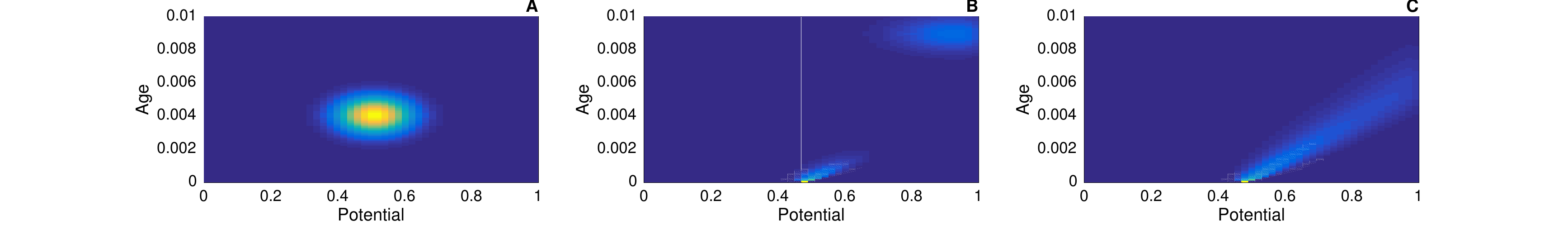}
   \caption{Simulation of the age-potential probability density function (\ref{model2'})-(\ref{icm1}).
    A two dimensional gaussian (age-potential) was taken as initial condition; the parameters of the simulation are:
   $v_r=0.5$, $\mu=20$, $\sigma =0.4$. The plots show the evolution in time of the solution, respectively $t=0$,
   $t=0.5$ $t=7$ for the respective panel A-B-C.  }\label{Fig07}
      \end{center}
\end{figure}

In Fig. \ref{Fig07}, a simulation of the age-potential pdf is shown.
The different panels (A-B-C) show the evolution in time of the density. The simulation starts
with a Gaussian as initial condition (the first panel of Fig. \ref{Fig07}). Under the influence of the drift term,
the density function advances in  age, which is clearly seen in the plots
of Fig. \ref{Fig07}. After the spiking process, the age of the neuron is reset to zero and its membrane potential is reset to $v_r$.
This effect is well perceived
 in the second panel of Fig. \ref{Fig07}.

\subsection{Theoretical connections}

We are ready to prove now the main result of our paper which shows that both models reminded in the previous sections,
i.e. FP model (\ref{FP})-(\ref{FPic}) respectively the AS model (\ref{AS})-(\ref{ASic}), can be obtained from our
two-dimensional system. We stress that the AS system derived from the age-and-potential system (\ref{model2'})-(\ref{icm1})
has a special age-dependent death rate, as it can be seen below.

\subsubsection{Derivation of the Fokker-Planck model}
The following result takes place:
\begin{theorem}\label{DFP}
Let $\pi$ a solution to the system (\ref{model2'})-(\ref{icm1}). Then, the probability density function defined as
\be\label{potential}
p(t,v)=\int_{0}^\infty \pi(t,a,v)\dd a,
\ee
is the solution to the FP system (\ref{FP})-(\ref{FPic}) with the initial repartition given by
\bd
p_0(v)=\int_0^\infty \pi_0(a,v)\dd a.
\ed
\end{theorem}
\emph{Remark.} Here and below, the solutions are defined in distributional sense. We did choose though to keep the computations at a formal level so that to render the reading easy.
\begin{proof}
The proof is simple and straightforward. Let us integrate the equation (\ref{model2'}) with
respect to $a$ over the age interval:

\bda
&&\int_0^\infty\frac{\partial}{\partial t} \pi(t,a,v)\dd a+
\int_0^\infty \frac{\partial}{\partial a}\pi(t,a,v)\dd a\\
&&+\int_0^\infty \frac{\partial}{\partial v}[(\mu(t)-v) \pi(t,a,v)]\dd a-
\int_0^\infty \frac{\sigma^2}{2} \frac{\partial^2 }{\partial v^2}\pi(t,a,v)\dd a=0,
\eda
which implies, by using the boundary condition at $a\lra\infty$:
\bda
&&\frac{\partial}{\partial t} \int_0^\infty\pi(t,a,v)\dd a-
\pi(t,0,v)+
\frac{\partial}{\partial v}\left[(\mu(t)-v) \int_0^\infty \pi(t,a,v)\dd a\right]\\
&&-\frac{\sigma^2}{2} \frac{\partial^2 }{\partial v^2}\int_0^\infty \pi(t,a,v)\dd a=0.
\eda
Now, using the boundary condition (\ref{fr})
and denoting by
$$p(t,v)=\int_0^\infty \pi(t,a,v)\dd a, $$
it
follows that, indeed, $p$ is solution to

\bd
\frac{\partial}{\partial t} p(t,v)+
\frac{\partial}{\partial v}[(\mu(t)-v) p(t,v)]-
\frac{\sigma^2}{2} \frac{\partial^2 }{\partial v^2}p(t,v)=\delta(v-v_r)r(t),
\ed
where
\bda
r(t)&=&\int_0^\infty \rho(t,a)\dd a=\int_0^\infty -\frac{\sigma^2}{2}\frac{\partial}{\partial v}\pi(t,a,1)\dd a\\
&=&-\frac{\sigma^2}{2}\frac{\partial}{\partial v}\int_0^\infty\pi(t,a,1)\dd a\\
&=&-\frac{\sigma^2}{2}\frac{\partial}{\partial v}p(t,1).
\eda
The boundary condition is obtained by:
\bd
p(t,1)=\int_0^\infty \pi(t,a,1)\dd a=0,
\ed
and the initial condition is defined as
\bd
p_0(v)=p(0,v)=\int_0^\infty \pi_0(a,v)\dd a.
\ed
We end the proof by noticing that, since the conservation property of the probability density $\pi(t,a,v)$ is assumed,
we get that
\bd
\int_{-\infty}^1 p(t,v)\dd v=1,
\ed
as soon as
\bd
\int_{-\infty}^1 p_0(v)\dd v=1,
\ed
is fulfilled.
The proof is now complete.

\end{proof}

\subsubsection{Derivation of the age-structured model}

The following result takes place:
\begin{theorem}\label{DAS}
Let $\pi$ a solution to the system (\ref{model2'})-(\ref{icm1}). Then, the probability density function defined as
\be\label{age}
n(t,a)=\int_{-\infty}^1 \pi(t,a,v)\dd v,
\ee
is the solution to the AS system (\ref{AS})-(\ref{ASic}) with the initial repartition given by
\bd
n_0(a)=\int_{-\infty}^1 \pi_0(a,v)\dd v.
\ed
\end{theorem}

\begin{proof}

 By integrating (\ref{model2'}) with respect to $v$ over the respective domain, one gets:
 \be\label{a1}
 \frac{\partial}{\partial t}\int_{-\infty}^1 \pi(t,a,v)\dd v+ \frac{\partial}{\partial a}\int_{-\infty}^1 \pi(t,a,v)\dd v+
 \rho(t,a)=0
 \ee
 where, as defined above,
 \bd
 \rho(t,a)=-\frac{\sigma^2}{2}\frac{\partial}{\partial v}\pi(t,a,1).
 \ed
The conservation property imposed on the system in $\pi$ will necessary lead, when integrating (\ref{a1}) with respect to $a$ over
the age domain, to:

\be
n(t,0)=\int_0^\infty \rho(t,a)\dd a.
\ee
As before, the initial condition is defined as

\be\label{ica}
n(0,a)=\int_{-\infty}^1 \pi_0(a,v)\dd v.
\ee

The system (\ref{a1})-(\ref{ica}) is an equivalent form of the system introduced in the first section. To see this,
there should be reminded the significance of the terms in the initial system.

We start by reminding that $S(t,a)$ was defined as the decay rate of the survivor function, which, in turn was
defined as the integral over whole potential space of a neuron that had the last spike at $t-a$; notice that, in terms
of our system, this is nothing else than $n(t,a)$. We therefore can force the notation into our system,
and write down:
\bd
S(t,a)=\frac{\rho(t,a)}{n(t,a)},
\ed
to arrive to the popularized version of the equation that gives the evolution with respect to the time
elapsed since the last spike:
\be\label{a3}
\frac{\partial}{\partial t}n(t,a)+\frac{\partial}{\partial a}n(t,a)+ S(t,a) n(t,a) =0
\ee
where
\be\label{fra2}
n(t,0)=\int_0^\infty S(t,a)n(t,a)\dd a.
\ee
The proof is now complete.

\end{proof}

\subsection{The stationary system}

Considering the probability that a neuron reaches a potential value $v$ at time $a$, given that at the initial time was in the
state $v_r$, this probability is shown to satisfy the following problem \cite{risken84, tuckwell88}:
\be\label{fpt}
 \frac{\partial}{\partial a}\varphi(a,v)+ \frac{\partial}{\partial v}[(\mu-v)\varphi(a,v)]
 -\frac{\sigma^2}{2}\frac{\partial^2}{\partial v^2}\varphi(a,v)=0,
 \ee
 with the initial condition given by
 \be
 \varphi(0,v)=\delta(v-v_r).
 \ee
The model is completed by an absorbing boundary condition at the firing threshold
 \be
 \varphi(a,1)=0,
 \ee
 and a reflecting boundary condition at the lower boundary of the domain:
\bea\label{1}
\lim\limits_{v \to -\infty} (-\mu +v)\varphi(a,v) +\frac{ \sigma ^2}{2}\frac{\partial }{\partial v}\varphi(a,v) =0.
\eea

This problem, known as the first passage time problem, can be therefore thought to express the neuronal probability
density function's dynamics on an inter-spike interval. It is due to this reason that we have kept here the notation
$a$ for the time variable, since, in the context introduced above, this time variable is interpreted as the time elapsed
since the last spike. All these considerations are made under the assumption that the stimulus $\mu$ is time-independent.
As it is well known, the flux at the firing threshold
\be\label{isi}
ISI(a)=-\frac{\sigma^2}{2}\frac{\partial}{\partial v}\varphi(a,1),
\ee
has the interpretation of the inter-spike distribution density function,
and the corresponding {\em survivor function} is given by:
\be\label{sf}
P(a)=\int_{-\infty}^1 \varphi(a,v)\dd v.
\ee

Denoting now by $\bar \pi$ the solution to the corresponding stationary system to the model (\ref{model2'})-(\ref{icm1}), i.e.
\be\label{stationary}
 \frac{\partial}{\partial a}\bar\pi(a,v)+ \frac{\partial}{\partial v}[(\mu-v)\bar\pi(a,v)]
 -\frac{\sigma^2}{2}\frac{\partial^2}{\partial v^2}\bar\pi(a,v)=0,
 \ee
 \be
 \bar\pi(0,v)=r\delta(v-v_r),
 \ee
 \be
\bar\pi(a,1)=0,
 \ee
\bea
\lim\limits_{v \to -\infty} (-\mu +v)\bar\pi(a,v) +\frac{ \sigma ^2}{2}\frac{\partial }{\partial v}\bar\pi(a,v) =0,
\eea
we obtain that it is given by
\be\label{stat}
\bar\pi(a,v)= r \varphi(a,v)
\ee
with
\bd
r=\int_0^\infty \rho(a)\dd a=-\frac{\sigma^2}{2}\frac{\partial}{\partial v}\bar \pi(a,1),
\ed
the stationary firing rate.

It is easy to see that, by integrating (\ref{stat}) with
respect to $v$ and $a$, respectively, on the corresponding state spaces, we obtain that
the distributions defined by:
\bd
\bar n(a)=\int_{-\infty}^1 \bar \pi(a,v)\dd v= rP(a),
\ed
respectively,
\bd
\bar p(v)=\int_0^\infty \bar \pi(a,v)\dd a=r\int_0^\infty \varphi(a,v)\dd a,
\ed
satisfy the corresponding stationary systems to AS and FP models.

In the case of the FP and AS models (\ref{FP})-(\ref{FPic}), respectively, (\ref{AS})-(\ref{ASic}), with
a time-independent stimulus $\mu$, we have found in \cite{usdt, usit} a way to map the solutions one-to-another
via integral transforms.

In this particular case, the following result takes place:

\begin{theorem}\label{homogeneous}
Let us consider the system (\ref{model2'})-(\ref{icm1}) with a constant in time stimulus $\mu(t)\equiv\mu$. Then
the system admits a separable solution
\be\label{sepsol}
\pi(t,a,v)=\frac{\varphi(a,v)n(t,a)}{P(a)},
\ee
where $\varphi$ is the solution to (\ref{fpt})-(\ref{1}) and $n(t,a)$ is the solution to
\be\label{TN}
\frac{\partial}{\partial t} n(t,a)+\frac{\partial}{\partial a} n(t,a)+S(a)n(t,a)=0
\ee
with
\be
n(t,0)=\int_0^\infty S(a)n(t,a)\dd a.
\ee
Here,
\bd
S(a)=\frac{ISI(a)}{P(a)}
\ed
with $ISI(a)$ and $P(a)$ defined by (\ref{isi}), (\ref{sf}) respectively.
\end{theorem}

It is easy to check that the function defined by (\ref{sepsol}) is indeed a solution to the system (\ref{model2'})-(\ref{icm1}).
Moreover, by integrating this function with respect to $a$ over the respective state space,
we recover the transform of the solution to AS model into the solution of the
FP model in case of time independent stimuli \cite{usdt}.
The recovery of the solution to the AS system in this case is trivial, reducing it by integration
of (\ref{sepsol}) with respect to v over the respective state space, to the the tautology $n=n$, by taking into
account that the survivor function $P$ is given by (\ref{sf}).

\subsection{A special solution for the non-autonomous case}
In fact, the result from the previous subsection can be adapted for the time dependent
stimulus case, given that a corresponding first passage time problem is correctly defined.

Let us define, therefore, the probability density for a neuron to be at time $t$ in the
state $(a,v)$, given that at the firing time $t-a$ it was in the state $(0,v_r)$ and denote it by $\varphi(t,a,v)$. Then, the corresponding FP equation is given again by (\ref{model2'}), but with a different boundary condition:
\be\label{bcfp}
\varphi(t,0,v)=\delta(v-v_r).
\ee
The rest of the boundary conditions are the same as above:
\bda
&&\lim\limits_{v \to -\infty} (-\mu(t) +v)\varphi(t,a,v)
+\frac{ \sigma ^2}{2}\frac{\partial }{\partial v}\varphi(t,a,v) =0,\\
\eda
\bd
\varphi(t,a,1) =0.\\
\ed
As initial condition, we consider
\bd
\varphi(0,a,v)=\varphi_0(a,v)
\ed
with $\varphi_0(a,v)$ solution to the first passage time problem for the constant in time stimulus ($\mu\equiv\mu(0)$) (\ref{fpt})-(\ref{1}).\\
The survivor function is then by definition:
\be
P(t,a)=\int_{-\infty}^1 \varphi(t,a,v)\dd v,
\ee
and the corresponding ISI function is given by
\be
ISI(t,a)=-\frac{\sigma^2}{2}\frac{\partial}{\partial v}\varphi(t,a,1).
\ee
It checks out immediately that
\bd
P(t,0)=1,
\ed
which is natural given the interpretation of the survivor function.
A corresponding relation between the survivor function and the ISI function takes place,
namely,
\be
ISI(t,a)=-\frac{\rm{D}}{\rm{Dt}} P(t,a)=-\left(\frac{\partial}{\partial t}P(t,a)+
\frac{\partial}{\partial a}P(t,a)\right),
\ee
where $\frac{\rm{D}}{\rm{Dt}}$ is the total derivative with respect to $t$.
The hazard rate is then defined as
\be\label{hr2a}
S(t,a)=\frac{ISI(t,a)}{P(t,a)}.
\ee
Let us consider now the system (\ref{a3})-{\ref{fra2}), this time with the hazard rate
defined by (\ref{hr2a}). Then, the following result holds:
\begin{theorem}\label{Transf}
The function defined as
\be
\pi(t,a,v)=\frac{\varphi(t,a,v)}{P(t,a)}n(t,a)
\ee
is a solution to the system (\ref{model2'})-(\ref{fr}), where $n(t,a)$ is a solution to
(\ref{a3})-(\ref{fra2}).
\end{theorem}
\begin{proof}
The proof is, again, straightforward. Let us formally check that equation (\ref{model2'}) is satisfied:
\bda
&&\frac{\partial}{\partial t}\left(\frac{\varphi(t,a,v)}{P(t,a)}n(t,a)\right)
+\frac{\partial}{\partial a}\left(\frac{\varphi(t,a,v)}{P(t,a)}n(t,a)\right)\\
&&+\frac{\partial}{\partial v}\left[(\mu(t)-v)\frac{\varphi(t,a,v)}{P(t,a)}n(t,a)\right]
-\frac{\sigma^2}{2}\frac{\partial^2}{\partial v^2}\left(\frac{\varphi(t,a,v)}{P(t,a)}n(t,a)\right)\\
&&=\left(\frac{\partial}{\partial t}\varphi(t,a,v)+\frac{\partial}{\partial a}\varphi(t,a,v)\right)\frac{n(t,a)}{P(t,a)}\\
&&+\left(\frac{\partial}{\partial t}n(t,a)+\frac{\partial}{\partial a}n(t,a)\right)\frac{\varphi(t,a,v)}{P(t,a)}\\
&&-\frac{\left(\frac{\partial}{\partial t}P(t,a)+\frac{\partial}{\partial a}P(t,a))\right)
}{P(t,a)}\frac{\varphi(t,a,v)}{P(t,a)}\\
&&+\frac{\partial}{\partial v}\left[(\mu(t)-v)\varphi(t,a,v)\right]\frac{n(t,a)}{P(t,a)}
-\frac{\sigma^2}{2}\frac{\partial^2}{\partial v^2}\varphi(t,a,v)\frac{n(t,a)}{P(t,a)}\\
&&=\frac{n(t,a)}{P(t,a)}\left(\frac{\partial}{\partial t}\varphi(t,a,v)+\frac{\partial}{\partial a}\varphi(t,a,v)+\frac{\partial}{\partial v}[(\mu(t)-v)\varphi(t,a,v)]
-\frac{\sigma^2}{2}\frac{\partial^2}{\partial v^2}\varphi(t,a,v)\right)\\
&&-\frac{\varphi(t,a,v)}{P(t,a)}\left(\frac{\partial}{\partial t}n(t,a)+\frac{\partial}{\partial a}n(t,a)+S(t,a)n(t,a)\right)\\
&&=0.
\eda
The boundary condition at $a=0$ checks out immediately:
\bd
\pi(t,0,v)=\frac{\varphi(t,0,v)}{P(t,0)}n(t,0)=\delta(v-v_r)r(t),
\ed
where
\bda
r(t)&=&\int_0^\infty S(t,a)n(t,a)\dd a=\int_0^\infty \frac{n(t,a)}{P(t,a)}\left(-\frac{\sigma^2}{2}\frac{\partial}{\partial v}\varphi(t,a,1)\right)\dd a\\
&=&\int_0^\infty \rho(t,a)\dd a.
\eda

\end{proof}

Note that Theorem \ref{Transf} does allow an actual integral transform of the solution to the AS system into the solution of the FP equation, similar to those obtained in \cite{usdt,usit}:
\be
p(t,v)=\int_0^\infty\frac{\varphi(t,a,v)}{P(t,a)}n(t,a)\dd a,
\ee
if we assume that initial repartition is given by
\bd
\pi_0(a,v)=\frac{\varphi_0(a,v)}{P(a)}n_0(a),
\ed
with $\varphi_0$ defined above,
so that:
\bd
p_0(v)=\int_0^\infty \pi_0(a,v)\dd a.
\ed
\section{Discussion}

The present paper is thought as a continuation of our previous work \cite{usdt,usit},
where we have uncovered a connection between the NLIF and the escape rate models in
the context of a time independent stimulus. We proved  there that the corresponding solutions to the FP equation and the AS
 system can be mapped one-to-another via integral transforms. In the present paper, we wish to go further by
 considering stimuli that are time dependent. Toward this end, we introduced a new stochastic model and we showed
 that starting from it, it is possible to recover both the NLIF and the escape rate models.

The importance of our finding is mostly theoretical. Our conclusion is that it is always possible to adjust parameters
such that the NLIF and the escape rate models to exhibit the same statistics. This finding enables us to find one
possible explanation for the reported similarities between the two frameworks \cite{plesser}.
The choice of using any of it would then be influenced only by which of
the variables - age or membrane potential - one wants to take into account.

Although the FP equation in neuronal dynamics context has been studied over the past few years and qualitative results regarding its solutions have
 been proven \cite{carillo01, csfp}, the AS systems theory has received a lot of attention
 in the last decades, see monographs \cite{iannelli, webb} for example,
 and different associated control problems have been considered \cite{anita}. There is, therefore, a big
 advantage from a mathematical point of view to privilege its use in neuronal modeling.

Let us finally conclude by saying that our work opens several pathways for future research. From our side, the
most exciting one is the investigation of emergent properties of neural networks. Both formalisms (FP and AS)
have been employed to describe neural circuits in the mean-field approximation (see \cite{B02, BH01, carillo01} for the FP and see \cite{MvV, chevalier,perthame03, perthame04} for the AS), and both of them have
generated key insights in neuroscience, especially in the quest of the underlying mechanism of neural
synchronization \cite{DH01,perthame04}. It would therefore be relevant to use our new framework to investigate the links and
differences between the two approaches, especially when conclusions are contradictory.

\bibliographystyle{plain}
\bibliography{source_system}

\end{document}